\documentclass[12pt]{article}

\newif\ifarXiv
\arXivtrue

\ifarXiv
\usepackage{e-jc_new}
\else
\usepackage{e-jc}
\fi

\usepackage{amsthm,amsmath,amssymb}

\usepackage{mathtools}
\usepackage{rotating}
\graphicspath{{figures/}}

\usepackage{graphicx}

\usepackage[colorlinks=true,citecolor=black,linkcolor=black,urlcolor=blue]{hyperref}

\theoremstyle{plain}
\newtheorem{theorem}{Theorem}[section]

\newtheorem{lemma}[theorem]{Lemma}

\newtheorem{proposition}[theorem]{Proposition}

\theoremstyle{definition}

\theoremstyle{remark}

\renewcommand{\emptyset}{\varnothing}

\title{\bf Note on the subgraph component polynomial}

\author{
Yunhua Liao\\
Department of Mathematics \\
Hunan Normal University,  Changsha, 410081, Hunan, China\\
\and
Yaoping Hou$\footnote{Corresponding author: yphou@hunnu.edu.cn}$  \\
Department of Mathematics \\
Hunan First Normal University,  Changsha, 410205, Hunan, China
}

\begin{document}

\maketitle


\begin{abstract}
Tittmann, Averbouch and Makowsky [P. Tittmann, I. Averbouch, J.A. Makowsky, The enumeration of vertex induced subgraphs with respect to the number of components, European Journal of Combinatorics, 32 (2011) 954-974], introduced the subgraph component polynomial $Q(G;x,y)$ which counts the number of connected components in vertex induced subgraphs. It contains much of the underlying graph's structural information, such as  the order, the size, the independence number. We show that there are  several other graph invariants, like the connectivity, the number of cycles of length four in a regular bipartite graph, are determined by the subgraph component polynomial.  We prove that several well-known families of graphs  are determined by the polynomial $Q(G;x,y).$ Moreover, we study the distinguishing power and find some simple graphs which are not distinguished by the subgraph component polynomial but distinguished by the character polynomial, the matching polynomial or the Tutte polynomial. These are answers to three problems  proposed by Tittmann et. al.

\end{abstract}

\section{Introduction}

 All graphs in this paper are simple and finite. Let $G=(V(G),E(G))$ be a graph. The \emph{order} and the \emph{size} of $G$ denote the number of vertices and the number of edges of $G$, respectively. The \emph{complete graph}, the \emph{cycle}, and the \emph{path} of order $n$, are denoted by $K_n$, $C_n$ and $P_n$, respectively. We denote the \emph{complete bipartite graph} with part sizes $m$ and $n$, by $K_{m,n}$. Also, $K_{1,n}$ is called a \emph{star}. Given a vertex $v \in V(G)$, the \emph{open neighborhood} of  vertex $v$ is denoted $N(v)$ and the \emph{closed neighborhood} is denoted $N[v]$. The degree of $v$ is the number of edges incident with $v$ and is denoted by $d(v)$. A \emph{pendant vertex} is a vertex of degree one. We denote the minimum degree of the vertices of $G$ by $\delta (G)$. If $U\subseteq V(G)$ is a vertex subset, then we denote the vertex induced subgraph of $G$ on $U$ by $G[U]$. An \emph{independent set} in a graph is a set of vertices no two of them are adjacent. The \emph{independence number} $\alpha(G)$ is defined as the cardinality of a maximum independent set in graph $G$.

A graph $G$ is \emph{connected} if any two of its vertices are linked by a path. A \emph{separating set} of a connected graph $G$ is a set of vertices whose removal renders $G$ disconnected. The \emph{connectivity} $c(G)$ (where $G$ is not a complete graph) is the order of a minimal separating set.  Obviously, $c(G) \leqslant \delta (G)$. A graph is called \emph{k-connected} if its vertex connectivity is  not less than $k$. This means if a graph $G$ is $k$-connected, then $G[ V\backslash U] $ is connected for every subset $U \subseteq V(G)$ with $|U|< k$. Complete graph $K_n$ has no separating set at all, but by convention $c(K_n)= n-1$. We say that a connected graph $G$ is a \emph{unicycle} if $|V(G)|=|E(G)|$. We can regard a unicycle as a cycle attached with each vertex a (rooted) tree. Further details of many of the concepts treated here can be found in Diestel \cite{Diestel00}.

There are several well-known graph polynomials, e.g. the Tutte polynomial \cite{Tutte54,Ellis11,Bernardi08}, the matching polynomial \cite{Farrell79,Dong12,Yan09}, the domination polynomial \cite{Akbari10,Kotek12} and the edge elimination polynomial \cite{Averbouch10,Trinks12}.

Recently, Tittmann, Averbouch and Makowsky \cite{Tittmann11} introduced a new graph polynomial, the subgraph component polynomial that is denoted by $Q(G;x,y)$, which counts the number of connected components in induced subgraphs. Like the bivariate Tutte polynomial, the polynomial $Q(G;x,y)$ has several remarkable properties, e.g. it is universal with respect to vertex elimination.  Tittmann et al. showed that the order, the size, the number of components and the independence number can be determined by $Q(G;x,y)$. In addition, they find the star graph $K_{1,n}$ is determined by $Q(G;x,y)$.  A number of open problems concerning the subgraph component polynomial  were posed  in their paper.
In this paper, we are mainly concerned with the following three problems
  in \cite{Tittmann11}.

\vspace*{2mm}
{\bf Problem 1.1.} {\em
Are there simple graphs distinguished by $p(G;x)$, $m(G;x)$, $P(G;x,y)$ or $T(G;x,y)$ which are not distinguished by $Q(G;x,y)?$}
\vspace*{2mm}

 {\bf Problem 1.2.} {\em Find more graph invariants which are determined by $Q(G;x,y)$.}
\vspace*{2mm}

 {\bf Problem 1.3.} {\em  Find more classes of graphs which are determined by $Q(G;x,y)$.}
\vspace*{2mm}

The aim of this paper is provide some posivtive answers to above three problems  and our main findings are:
\begin{itemize}
  \item We discover much more information  hiding in the polynomial $Q(G;x,y)$, e.g. the connectivity $c(G)$ (Theorem~\ref{theorem:connectivity}), regularity (Proposition~\ref{proposition:regularity}). In particular, if $G$ is a regular bipartite graph, the number of cycles of length four can be determined by $Q(G;x,y)$ (Theorem~\ref{theorem:4cycle}), while it is a well-known fact that this parameter is also determined by the Tutte polynomial \cite{Mier04}.
  \item We find several classes of graphs which are determined by the subgraph component polynomial, e.g. the path $P_n$, the cycle $C_n$, the complete bipartite graph $K_{m,n}$, the friendship graph $C_3^n$, the book graph $B_n$ and the $n$-cube $Q^n$ (Section 4).
  \item We find two simple graphs distinguished by the character polynomial, the matching polynomial or the Tutte polynomial which are not distinguished by the subgraph component polynomial (Proposition~\ref{proposition:power}).
\end{itemize}
%
%
   \section{The subgraph component polynomial}
The subgraph component polynomial $Q(G; x, y)$ of a graph  was introduced  by P. Tittmann, I. Averbouch and J. A. Makowsky in \cite{Tittmann11}, and have been further studied by Garijo et al. in \cite{Garijo11,Garijo13}. This graph polynomial arises from analyzing community
structures in social networks. Let $k\left( G\right) $ be the number of components of $G$, and let $q_{i,j}\left( G\right) $ be the number of vertex subsets $X\subseteq V$ with $i$ vertices such that $G\left[ X\right] $ has exactly $j$ components:
\[
q_{i,j}\left( G\right) =\left\vert \left\{ X\subseteq V:\left\vert
X\right\vert =i\wedge k\left( G\left[ X\right] \right) =j\right\}
\right\vert.
\]
The \emph{subgraph component polynomial} of $G$ is defined as an ordinary generating function for these numbers:
\[
Q\left( G;x,y\right) =\sum_{i=0}^{n}\sum_{j=0}^{n}q_{i,j}\left( G\right)
x^{i}y^{j}.
\]
If we sum over all the possible subsets of vertices, the definition can be rewritten in a slightly different way:
\[
Q\left( G;x,y\right) =\sum_{X\subseteq V}
x^{|X|}y^{k(G[X])}.
\]
P. Tittmann et al. defined three types of vertex elimination operations on graphs:

\begin{itemize}
  \item \emph{Deletion}. $G-v$ denote the graph obtained by simply removing the vertex $v$.
  \item \emph{Extraction}. $G-N[v]$ denote the graph obtained from $G$ by removal of all vertices adjacent to $v$ including $v$ itself.
  \item \emph{Contraction}. $G/v$ denote the graph  obtained from $G$ by removal of $v$ and insertion of edges between all pairs of non-adjacent neighbor vertices of $v$.
\end{itemize}
They showed that $Q(G;x,y)$ satisfies the following linear recurrence relation with respect to this three kinds of vertex elimination operations and is universal in this respect.
\begin{proposition}\label{proposition:elimination}{\bf \cite{Tittmann11}}
Let $G=(V,E)$ be a graph and $v\in V$. Then the subgraph component polynomial satisfies the decomposition formula
\begin{equation*}
Q(G;x,y)=Q(G-v;x,y)+x(y-1)Q(G-N[v];x,y)+xQ(G/v;x,y).
\end{equation*}
\end{proposition}
The previous definition of $q_{i,j}$ yields the following proposition.
\begin{proposition} \label{proposition:unimodal}
If $H$ is subgraph of $G$, then $q_{i,j}(H)\leqslant q_{i,j}(G)$.
\end{proposition}

In what follows, we will call two graphs $G$ and $H$ are  \emph{subgraph component equivalent}, or simply $Q$-equivalent, if $Q(G;x,y)=Q(H;x,y)$.  A graph $G$ is $Q$-unique if any graph $H$ being $Q$-equivalent to $G$ implies that $H$ is isomorphic to $G$. Let $[x^iy^j]Q(G;x,y)$ denote the coefficient of $x^iy^j$ in $Q(G;x,y)$, and let $deg_xQ(G;x,y)$ be the degree with respect to the variable $x$.
\section{Graph invariants determined by the subgraph component polynomial}
\begin{proposition} {\bf \cite{Tittmann11}}
\label{prop:number}
The following graph properties can be easily obtained from the subgraph
polynomial:
\begin{enumerate}
\item[(1)]
The number of vertices:%
\[
|V(G)| =\deg _{x}Q\left(
G;x,y\right) =\log _{2}Q\left( G;1,1\right)=\left[ xy\right] Q\left( G;x,y\right)
\]

\item[(2)]
The number of edges:%
\[
|E(G)| =\left[ x^{2}y\right] Q\left( G;x,y\right)
\]

\item[(3)]
The number of components:%
\[
k\left( G\right) =\deg_{y} \left( \left[ x^{|G| }\right] Q\left(
G;x,y\right) \right)
\]
\item[(4)]
The number of independent sets of each size; in particular, the independence number:
\[
\alpha \left( G\right)=\deg _{y}Q\left( G;x,y\right)
\]
\end{enumerate}
\end{proposition}


Since the order, the size and the number of components of a graph $G$ is determined by its subgraph component polynomial $Q(G;x,y)$, it is clear that if $G$ is a tree and $H$ is $Q$-equivalent to $G$, then $H$ is also a tree.
\begin{theorem}\label{theorem:connectivity}
The connectivity of a graph $G$ is  determined by its subgraph component polynomial.
\end{theorem}
\begin{proof}
As we shall show that complete graphs are determined by the subgraph component polynomial in the next section, we just consider graphs which are not complete here. Let $G=(V,E)$ be a graph of order $n$. Given a vertex subset $S$ with cardinality $s$, then $k(G[V\backslash S])\geqslant 2$ when $S$ is a separating set of $G$. From the definition of the connectivity $\kappa(G)$, we have the following relation:
\begin{eqnarray*}
 c(G)&=&min \left\{s:\left[x^{n-s}y^{j} \right]Q\left( G;x,y \right)> 0 \wedge j \geqslant 2 \right\}\\
                  &=&n-max\left \{deg_{x}\left(\left [y^j \right]Q\left(G;x,y \right)\right): j\geqslant 2 \right\}.
\end{eqnarray*}
\end{proof}

Since the connectivity is  a lower bound for the minimum degree $\delta(G)$, the order and the size of graph $G$ are determined by $Q(G;x,y)$ as well, then we can conclude that the regularity of graph $G$ is also determined by $Q(G;x,y)$.

\begin{proposition}
\label{proposition:regularity}
Let $G$ be a $k$-regular graph. If $H$ is $Q$-equivalent to $G$, then $H$ is $k$-regular.
\end{proposition}

\begin{theorem}  \label{theorem:4cycle}
Let $G$ be a $k$-regular bipartite graph of order $n$. Then the number of subgraphs isomorphic to $C_4$ and the number of subgraphs isomorphic to $P_4$ are determined by $Q(G;x,y)$.
\end{theorem}
\begin{proof}
There are three  subgraphs, paths of order 4, cycles of order 4 and complete bipartite subgraphs $K_{1,3}$, which contribute to $[x^4y]Q(G;x,y)$ (see Fig.~\ref{fig.1})

\begin{figure}[h]
\begin{center}
\includegraphics[width=0.6\linewidth]{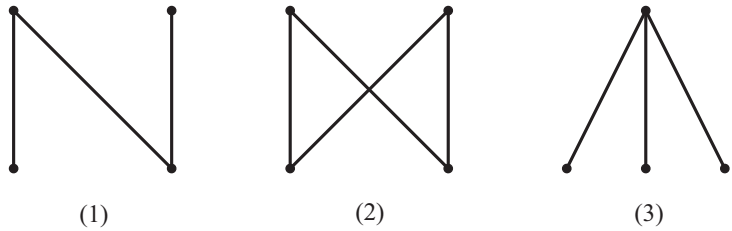}
\end{center}
\caption{All possible subgraphs in $G$ with $4$ vertices and 1 component.}
\label{fig.1}
\end{figure}
Since $G$ is $k$-regular, the contribution of  $K_{1,3}$ is $n {k \choose 3}$. It follows that
\begin{equation}\label{equation:x4y}
\left [x^4y \right ]Q\left (G;x,y\right )=p+c+n{k \choose 3}
\end{equation}
where $p$ is the number of paths of order $4$ and $c$ is the number of cycles of order $4$.

For $G$ is a bipartite graph, a subgraph $A \subseteq G$ contribute to $[x^3y]Q(G;x,y)$ if and only if $A$ is isomorphic to $K_{1,2}$. Let $K_{1,2}^+$ be a subgraph obtained from $K_{1,2}$ by adding a new vertex $u$ which is adjacent to at least one of the two vertices $x_1,x_2$ in the partite set of cardinality two. Then $K_{1,2}^+$ is a $P_4$ if the new vertex $u$ is adjacent to only one vertex of $x_1,x_2$, and $K_{1,2}^+$ is a $C_4$ if $u$ is adjacent to both $x_1$ and $x_2$. We can count  the number of $K_{1,2}^+$ (counted with multiplicity) in two ways:  As $G$ is $k$-regular, for every subgraph $K_{1,2}$, $x_1$ and $x_2$ have $2k-2$ neighbor vertices (counted with multiplicity) other than the vertex in the partite set of cardinality one. On the other hand, every $P_4$ is counted twice and
every $C_4$ is counted $8$ times in the counting process. Then the following relation holds:
\begin{equation}\label{equation:x3y}
(2k-2)\left [x^3y \right ]Q\left (G;x,y\right )=2p+8c.
\end{equation}
 Equations (\ref{equation:x4y}) and (\ref{equation:x3y}) imply that  $p$ and $c$ can be obtained from the coefficients of $Q(G;x,y)$.
\end{proof}

\section{Graphs determined by their subgraph component polynomial}
The Tutte polynomial is a well-studied graph polynomials. In \cite{Mier04}, Mier and Noy show that wheels, squares of cycles, ladders, M\"{o}bius Ladders, complete multipartite graphs, and hypercubes are $T$-unique graphs. In \cite{Tittmann11}, it was shown that the star graph $S_n$ is $Q$-unique and  is a open problem (Problem 35 in \cite{Tittmann11} ) to find more classes of graphs  which  are determined by $Q(G;x,y)$. In this section we will show that complete graphs, paths, cycles, tadpole graphs, complete bipartite graphs, friendship graphs, book graphs and hypercubes are $Q$-unique, or are determined by the subgraph component polynomial.


\subsection{Complete graphs, paths, cycles, tadpole graphs and complete bipartite graphs.}
$K_n$ is a graph of order $n$ and size $n \choose 2$. Since the order and the size are both determined by the subgraph component polynomial, so does $K_n$.
\begin{proposition}
The complete graph $K_n$ is $Q$-unique.
\end{proposition}

The cycle graph $C_n$ is a 2-connected unicycle with $n$ vertices.
\begin{proposition}\label{theorem:cycle}
The cycle $C_n$ is $Q$-unique.
\end{proposition}
\begin{proof}
Let $H$ be $Q$-equivalent to the cycle $C_n$. Then $H$ is a 2-connected graph, $|V(H)|=|E(H)|=n$. It follows that $H\cong C_n$.
\end{proof}

\begin{theorem} \label{theorem:path}
The path $P_n$ is $Q$-unique.
\end{theorem}
\begin{proof}
Let $H$ be $Q$-equivalent to the path $P_n$. It is clear that $H$ is a tree of order $n$. Since removing a vertex in a path could increase the number of connected components by at most 1, $deg_y[x^{n-1}]Q(H;x,y)=deg_y[x^{n-1}]Q(P_n;x,y)\leqslant 2$.

{\bf Claim} There is at most two pendant vertices in $H$.

\begin{proof}
Suppose $H$ have at least three pendant vertices. Let $x,y,z$ be three pendant vertices of $H$. As $H$ is a tree, there is a unique path $P_y$ from $x$ to $y$, and we denote the unique path from $x$ to $z$ by $P_z$. Let $meet_{yz}$ be the last vertices in the intersection of $P_y$ and $P_z$. That is, the path from $meet_{yz}$ to $y$ and the path from $meet_{yz}$ to $z$ are disjoint except at $meet_{yz}$. Let $X=V(H)\backslash meet_{yz}$, then $|X|=n-1$ and $k(G[X])\geqslant 3$ which leads to $deg_y[x^{n-1}]Q(H;x,y)\geqslant 3$, a contradiction.
\end{proof}

As a tree have at least two pendant vertices, then $H$ has exactly two pendant vertices. Therefore, $H$ is a path  as the path  $P_n$ is the tree with  exact two pendant vertices.
\end{proof}

The $(m,n)$-tadpole graph $T_{m.n}$ is the graph obtained by joining a path graph $P_m$ to  a cycle graph $C_n$ with a bridge.
\begin{theorem}
The tadpole graph $T_{m,n}$ is $Q$-unique.
\end{theorem}
\begin{proof}
Let $H$ be $Q$-equivalent to $T_{m,n}$. Then $H$ is a 1-connected graph, $|V(H)|=|E(H)|=m+n$. It follows that $H$ is a unicycle. Similar to the proof of the Claim in Theorem~\ref{theorem:path}, we can prove that $H$ has  at most one pendant vertex. Since $H$ is a 1-connected graph, then $H$ can be constructed by joining one path to a cycle by a bridge. In addition, the order of the path is counted by $[x^{m+n-1}y^2]Q(H;x,y)$. This proves the theorem.
\end{proof}

A complete bipartite graph $K_{m,n}$ is a bipartite graph $(X,Y,E)$ such that for every two vertices $x \in X$ and $y \in Y$, $\{x,y\}$ is an edge in $K_{m,n}$.

\begin{theorem}
The complete bipartite graph $K_{m,n}$ is $Q$-unique.
\end{theorem}
\begin{proof}
Let $H$ be a graph with the same subgraph component polynomial as $K_{m,n}$, for some $m\geqslant n$. Then $H$ is a $n$-connected graph, $|H|=m+n$, $|E(H)|=mn$, $\alpha(H)=m$. For each vertex subset $A$ of cardinality $m+1$, $K_{m,n}[A]$ is connected, then $deg_y[x^{m+1}]Q(H;x,y)=deg_y[x^{m+1}]Q(K_{m,n};x,y)=1$. Let $X=\{x_1,x_2,\dots,x_m\}$ be a maximum independent set of $H$, $Y=V(H)\backslash X=\{y_1,y_2,\dots,y_m\}$. We claim that every vertex in $X$ adjacent to every vertex in $Y$. If not, we can assume that there are two vertices $x_i \in X$, $y_j \in Y$ such that $\{x_i,y_j\} \notin E(H)$. Let $Z=X\cup \{y_j\}$. Then $|Z|=m+1$ and $k(H[Z]) \geqslant 2$ which leads to $deg_y[x^{m+1}]Q(H;x,y) \geqslant 2$, a contradiction. Therefore, $|E(X,Y)|=mn=|E(H)|$. Hence $Y$ is an independent set of $H$, so $H\cong K_{m,n}$.
\end{proof}

\subsection{Friendship graphs, book graphs}
The friendship graph $C_3^n$ can be constructed by joining $n$ copies of the cycle graph $C_3$ with a common vertex $u$, see Fig.~\ref{friendship}. Wang $et$ $al$. found that the friendship graph $C_3^n$ can be determined by the signless Laplacian spectrum in \cite{Wang10}.
\begin{figure}[h]
\begin{center}
\includegraphics[width=0.3\linewidth]{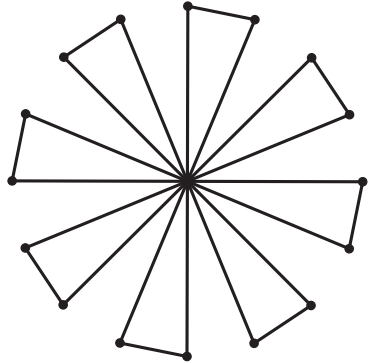}
\end{center}
\caption{The friendship graph $C_3^8$.}
\label{friendship}
\end{figure}
\begin{theorem}\label{theorem:friend}
The friendship graph $C_3^n$ is $Q$-unique.
\end{theorem}
\begin{proof}
Let $H$ be $Q$-equivalent to $C_3^n$. Then $H$ is a 1-connected graph, $|V(H)|=2n+1$, $|E(H)|=3n$, $\alpha(H)=n$.
  Since $[x^{2n}y^n]Q(H;x,y)=[x^{2n}y^n]Q(C_3^n;x,y)=1$, then there is a subgraph of $H$ with $2n$ vertices and $n$ components. We denote these components by $H_1,H_2,\cdots,H_n$. As $H$ is connected, there must be a vertex $u$ in $H$ such that $u$ is connected to each of  above $n$ components.

 {\bf Claim} Each component $H_i$ have exactly two vertices.

  \begin{proof}
We first claim that $|H_i|\geqslant 2$ for each $i$. If there exist $j$ such that $|H_j|=1$, let $H_j=\{x\}$. Then $V(H)\backslash \{u,x\}$ induces a subgraph in $H$ with $2n-1$ vertices and $n-1$ components; but such a subgraph does not exist in $C_3^n$. Since the total number of vertices is $2n+1$, we have $|H_i|=2$ for each $i$.
 \end{proof}

 For each $i$, let $V(H_i)=\{x_i,y_i\}$.   We first prove that neither $x_i$ nor $y_i$ is a pendant vertex. We suppose $x_i$ is a pendant vertex in $H$, then $y_i$ is the unique neighbor vertex of $x_i$ for $x_i,y_i$ constitute a component. Then $V(H)\backslash y_i$ induces a subgraph in $H$ has $2n$ vertices and $2$ components; but such a subgraph does not exist in $C_3^n$. Therefore $\{x_i,y_i,u\}$ introduce a cycle $C_3$ in $H$ and we complete the proof.
\end{proof}

The $n$-book graph $B_n$ can be constructed by joining $n$ copies of the cycle graph $C_4$ with a common edge $\{u,v\}$, see Fig.~\ref{book}.
\begin{figure}[h]
\begin{center}
\includegraphics[width=0.5\linewidth]{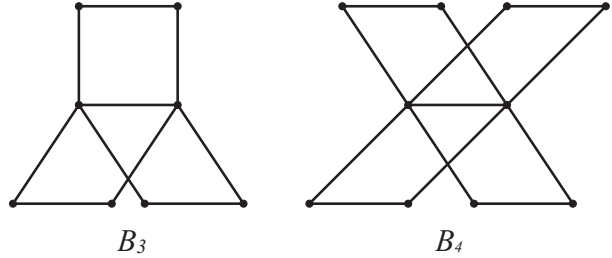}
\end{center}
\caption{The book graphs $B_3$ and $B_4$.}
\label{book}
\end{figure}
\begin{theorem}
The book graph $B_n$ is $Q$-unique.
\end{theorem}
\begin{proof}
Let $H$ be a graph $Q$-equivalent to $B_n$. Then $H$ is a 2-connected graph with $2n+2$ vertices and $3n+1$ edges and $\alpha(H)=n+1$. According the special structure of $B_n$, the following two equations holds:

\begin{equation}\label{equation:book1}
[x^{2n}y^n]Q(H;x,y)=[x^{2n}y^n]Q(B_n;x,y)=1,
\end{equation}
\begin{equation}\label{equation:book2}
[x^{n+1}y^{n+1}]Q(H;x,y)=[x^{n+1}y^{n+1}]Q(B_n;x,y)=2.
\end{equation}
It follows from Equation~(\ref{equation:book1}) that there are two vertices $u$ and $v$ in $V(H)$ such that $V(H)\setminus \{u,v\}$ induces a subgraph of $H$ with $2n$ vertices and $n$ components. We denote these components by $H_1,\cdots,H_n$. Similar to the proof for the Claim in Theorem~\ref{theorem:friend} we can show $H_i$ has exactly two vertices $x_i,y_i$, for each $i$. $H$ is a $2$-connected graph implies that for each $i$, there is one vertex in $H_i$ adjacent to $u$ and another vertex in $H_i$ adjacent to $v$. Without loss of generality we let $x_i$ adjacent to $u$ and $y_i$ adjacent to $v$. We suppose there is a vertex $y_k$ in the component $H_k$ such that $\{u,y_k\} \in H$. Since $H$ has exactly $3n+1$ edges, then $\{u,v\} \notin H$ and $\{v,x_i\} \notin H$ for each $i$. Therefore, $[x^{n+1}y^{n+1}]Q(H;x,y)=1$ which is a contradiction with equation~(\ref{equation:book2}). Hence, $\{u,y_i\} \notin H_i$ for each $i$. Analogously, $\{v,x_i\} \notin H_i$ for each $i$. So $ \{x_i,y_i,u,v\}$ constitute a cycle $C_4$ for each $i$ and we have finished the proof.
\end{proof}
\subsection{Hypercubes}
The \emph{n-cube} $Q^n$ is defined as the product of $n$ copies of $K_2$. Cubes have been extensively studied  in computer science, and the following two lemma for the case of vertex fault were shown in \cite{Esfahanian89,Yang04}.
\begin{lemma}\label{lemma:cube1}\cite{Esfahanian89,Yang04}
Let $F$ be a set of at most $2n-3$ vertices in $Q^n$. If $N(u) \varsubsetneq F$ for each vertex $u$ in $Q^n$, then $Q^n-F$ is connected.
\end{lemma}

\begin{lemma}\label{lemma:cube2}\cite{Yang04}
Let $u$ be a vertex in $Q^n$. Then $c(Q^n-N[u])=n-2$.
\end{lemma}
Lemma~\ref{lemma:cube1} and \ref{lemma:cube2} imply that in order to get a subgraph of $Q^n$ with more than two components, at least $2n-2$ vertices should be deleted.
\begin{lemma}\label{lemma:cube3}
Let $F$ be a set of at most $2n-3$ vertices in $Q^n$. Then $k(Q^n-F) \leqslant 2$.
\end{lemma}

From the proof of the Theorem 6.2 in \cite{Mier04}, we can conclude the following characterization of the $n$-cube.

\begin{lemma}\label{lemma:cube}
A connected $n$-regular graph is isomorphic to the $n$-cube if and only if it has $2^n$ vertices, $2^{n-2}{n \choose 2}$ subgraphs isomorphic to $C_4$ and no subgraph isomorphic to $K_{2,3}$.
\end{lemma}

We are now in a position to prove that $Q^n$ is $Q$-unique.

\begin{theorem}\label{theorem:cube}
The $n$-cube is $Q$-unique for every $n\geqslant 2$.
\end{theorem}
\begin{proof}
Let $H$ be $Q$-equivalent to the $n$-cube $Q^n$. Then $H$ is $n$-connected, $n$-regular, has $2^n$ vertices and $n2^{n-1}$ edges, $\alpha(H)=2^{n-1}$.

 {\bf Claim 1} For every $s\geqslant 1$, $deg_y[x^{2^n-s}]Q(H;x,y)\leqslant s$.

\begin{proof}
If graph $G$ is Hamiltonian, then $k(G-A)\leqslant |A|$ for every subset $A \subseteq V(G)$.  As well and long known, $Q^n$ is  Hamiltonian. Therefore,
for every vertex subset $A \subseteq V(Q^n)$ of cardinality $s$, $k(Q^n[V\backslash A])\leqslant s $. Then $deg_y[x^{2^n-s}]Q(H;x,y)=deg_y[x^{2^n-s}]Q(Q^{n};x,y) \leqslant s$.
\end{proof}

It is evident that $[x^{2^{n-1}}y^{2^{n-1}}]Q(H;x,y)=[x^{2^{n-1}}y^{2^{n-1}}]Q(Q^n);x,y)=2$. Let $X,Y$ be two separating sets of $H$ and  $|X|=|Y|=2^{n-1}$.

{\bf Claim 2} $X$ and $Y$ are disjoint.

\begin{proof}
If not, let us suppose $Z=X\cap Y$ and $|Z|=s\geqslant 1$. Let $X=\{x_1,\dots,x_{2^{n-1}-s},z_1,\dots,z_s\}$, $Y=\{y_1,\dots,y_{2^{n-1}-s},z_1,\dots,z_s\}$, $Z=\{z_1,\dots,z_s\}$ and $U=X\cup Y$. Then $|U|=2^n-s$ and $k(H[U])\geqslant s+1$, which leads to $deg_y[x^{2^n-s}]Q(H;x,y)\geqslant s+1$, a contradiction to Claim 1.
\end{proof}

Claim 2 implies that $H$ is a regular bipartite graph. Consequently, $H$ has $2^{n-2}{n \choose 2}$ subgraphs isomorphic to $C_4$. In view of Lemma \ref{lemma:cube}, we just need to prove that  $H$ has no subgraph isomorphic to $K_{2,3}$. Let $H=(V_1,V_2,E)$ with $|V_1|=|V_2|=2^{n-1}$  and $V_1\cap V_2=\emptyset$. For every pair of vertices $a$ and $b$ at distance $2$ in $H$, $a$ and $b$ belong to the same partite set. Let $n(a,b)$ be the number of common neighbor vertices they have.

{\bf Claim 3} $n(a,b)\leqslant 2$.

\begin{proof}
Without loss of generality we suppose there is a pair of vertices $a$ and $b$ at distance $2$ in $V_1$ such that $n(a,b)\geqslant 3$. Let $c_1,c_2,c_3$ are three vertices which are adjacent to both $a$ and  $b$. Then  $c_1,c_2,c_3\in V_2$. Since  $H$ is $n$-regular, we can let
\begin{eqnarray*}
N(a)&=&\{c_1,c_2,c_3,a_1,a_2,\dots,a_{n-3}\},\\
N(b)&=&\{c_1,c_2,c_3,b_1,b_2,\dots,b_{n-3}\},
\end{eqnarray*}
where $a_i$ and $b_i$ are neighbor vertices of $a$ and $b$, respectively. Let $N(a,b)=N(a)\cup N(b)$. Then $|N(a,b)| \leqslant 2n-3$ and $k(H[V\backslash N(a,b)])\geqslant 3$ which is  a contradiction with Lemma~\ref{lemma:cube3}.
\end{proof}

It follows from claim 3 that $H$ has no subgraph isomorphism to $K_{2,3}$.
\end{proof}

%

\section{Distinctive power}

We denote by $m(G;x)=\Sigma_{i=0}^n(-1)^im_ix^{n-2i}$ the matching polynomial with $m_i$ is the number of $i$-matchings in $G$, by $p(G;x)$ the characteristic polynomial, by $T(G;x,y)$ the Tutte polynomial. The Tutte polynomial does not distinguish 1-connected graphs and the subgraph component polynomial does not distinguish between graphs which differ only by the multiplicity of their edges. In this section we shall give a family of 2-connected simple graphs with the same Tutte polynomial but different subgraph component polynomials. Moreover, we find two  $Q$-equivalent simple graphs which can be distinguished by the character polynomial $p(G;x)$, the matching polynomial $m(G;x)$ or the Tutte polynomial $T(G;x,y)$.

The join $G\vee H$ of two graphs $G=(V,E)$ and $H=(W,F)$ with $V \cap W=\emptyset$ is the graph obtained from $G\cup H$ by introducing edges from each vertex of $G$ to each vertex of $H$.

The graph $K_1\vee P_n$ is called   fan graph $F_n$. In the fan $F_n$, the vertices corresponding to the path $P_n$ are
labeled from $v_1$ to $v_n$, and the central vertex corresponding to $K_1$ is labeled as $v_0$. $F_{n-1}^+$ arises from $F_{n-1}$ by adding a new vertex $v_n$ and two new edges $\{v_{n-2},v_n\},\{v_{n-1},v_n\}$, see Fig.~\ref{fig.2}.
\begin{figure}[h]
\begin{center}
\includegraphics[width=0.6\linewidth]{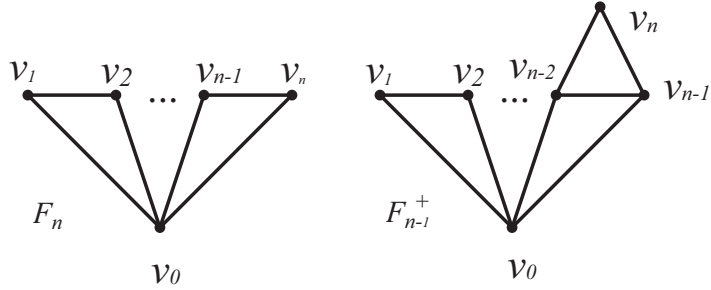}
\end{center}
\caption{The graphs $F_n$, $F_{n-1}^+$.}
\label{fig.2}
\end{figure}
\begin{proposition}
For $n\geqslant 5$, $T(F_n;x,y)=T(F_{n-1}^+;x,y)$ but $Q(F_n;x,y)\neq Q(F_{n-1}^+;x,y)$.
\end{proposition}

\begin{proof}
Observe that $F_n$ and $F_{n-1}^+$ have the same dual graphs. Since $T(G;x,y)=T(G^*;x,y)$ for a planar graph $G$ and its dual graph $G^*$, we have $T(F_n;x,y)=T(F_{n-1}^+;x,y)$. We note that $F_n-\{v_n\}=F_{n-1}^+-\{v_n\}=F_{n-1}$, $F_n/v_n=F_{n-1}^+/v_n=F_{n-1}$, $F_n-N[v_n]=P_{n-2}$ and $F_{n-1}-N[v_n]=F_{n-3}$. Since paths are $Q$-unique, then $Q(P_{n-2};x,y)\neq Q(F_{n-3};x,y)$. Proposition~\ref{proposition:elimination} implies that $Q(F_n;x,y)\neq Q(F_{n-1};x,y)$.
\end{proof}

\begin{figure}[h]
\begin{center}
\includegraphics[width=0.5\linewidth]{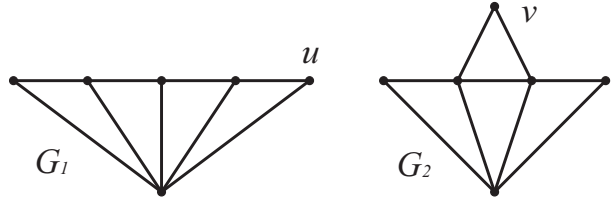}
\end{center}
\caption{The graphs $G_1$, $G_2$.}
\label{power}
\end{figure}
 \begin{proposition} For the graphs $G_1$ and $G_2$ illustrated in Fig.~\ref{power}, we have
 \label{proposition:power}

 (1)  $Q(G_1;x,y)=Q(G_{2};x,y)$.

 (2)  $p(G_1;x)\neq p(G_{2};x)$.

 (3)  $m(G_1;x)\neq m(G_{2};x)$.

 (4)  $T(G_1;x,y)\neq T(G_{2};x,y)$.

 \end{proposition}
 \begin{proof}
 We eliminate vertices $u$ and $v$ in graphs $G_1$ and $G_2$, respectively. It is not difficult to see that $G_1-u=G_2-v=F_4$, $G_1-N[u]=G_2-N[v]=P_3$ and $G_1/u=G_2/v=F_4$. Then $Q(G_1;x,y)=Q(G_{2};x,y)$. Using the graph package for Maple  we can  compute the characteristic polynomials, the matching polynomials, and the Tutte polynomials of $G_1$ and $G_2$ as follows:
\begin{eqnarray*}
p(G_1;x)&=&x^6-9x^4-8x^3+9x^2+8x-1,\\
p(G_2;x)&=&x^6-9x^4-8x^3+9x^2+6x-4,\\
 m(G_1;x)&=&x^6-9x^4+15x^2-2,\\
m(G_2;x)&=&x^6-9x^4+15x^2-3,
\end{eqnarray*}
\begin{eqnarray*}
 T(G_1;x,y)&=&x^5+4x^4+4x^3y+3x^2y^2+2xy^3+y^4+6x^3\\
             &+&9x^2y+7xy^2+3y^3+4x^2+6xy+3y^2+x+y,\\
T(G_2;x,y)&=&x^5+4x^4+4x^3y+3x^2y^2+3xy^3+y^4+6x^3\\
          &+&9x^2y+6xy^2+2y^3+4x^2+6xy+3y^2+x+y.
\end{eqnarray*}
\end{proof}

\section*{Acknowledgements}
This project was supported by the National Natural Science Foundation of China (No. 1171102) and the Hunan Provincial Innovation Foundation For Postgraduate (No. CX2013B214).

\end{document}

\end{document}